\def\draft{n}
\documentclass[11pt]{amsart}
\usepackage[headings]{fullpage}
\usepackage{amssymb,epic,eepic,epsfig,amsmath,amsbsy,amsmath,bbm}
\usepackage{graphicx}
\usepackage{url}
\usepackage[bookmarks=true,%
    colorlinks=true,%
    linkcolor=blue,%
    citecolor=blue,%
    filecolor=blue,%
    menucolor=blue,%
    urlcolor=blue,%
    breaklinks=true]{hyperref}

\newtheorem{theorem}{Theorem}[section]
\theoremstyle{definition}

\newtheorem{lemma}[theorem]{Lemma}
\newtheorem{definition}[theorem]{Definition}
\newtheorem{remark}[theorem]{Remark}
\newtheorem{corollary}[theorem]{Corollary}

\def\printname#1{
        \if\draft y
                \smash{\makebox[0pt]{\hspace{-0.5in}
                        \raisebox{8pt}{\tt\tiny #1}}}
        \fi
}

\newcommand{\psdraw}[2]
         {\begin{array}{c} \hspace{-1.3mm}
        \raisebox{-4pt}{\epsfig{figure=draws/#1.eps,width=#2}}
        \hspace{-1.9mm}\end{array}}

\newlength{\standardunitlength}
\catcode`\@=11
\long\def\@makecaption#1#2{%
     \vskip 10pt

\setbox\@tempboxa\hbox{
       \small\sf{\bfcaptionfont #1. }\ignorespaces #2}%
     \ifdim \wd\@tempboxa >\captionwidth {%
         \rightskip=\@captionmargin\leftskip=\@captionmargin
         \unhbox\@tempboxa\par}%
       \else
         \hbox to\hsize{\hfil\box\@tempboxa\hfil}%
     \fi}
\font\bfcaptionfont=cmssbx10 scaled \magstephalf
\newdimen\@captionmargin\@captionmargin=2\parindent
\newdimen\captionwidth\captionwidth=\hsize
\catcode`\@=12

\def\lbl#1{\label{#1}\printname{#1}}

\newdimen\tableauside\tableauside=1.0ex
\newdimen\tableaurule\tableaurule=0.4pt
\newdimen\tableaustep
\def\phantomhrule#1{\hbox{\vbox to0pt{\hrule height\tableaurule
width#1\vss}}}
\def\phantomvrule#1{\vbox{\hbox to0pt{\vrule width\tableaurule
height#1\hss}}}
\def\sqr{\vbox{%
  \phantomhrule\tableaustep

\hbox{\phantomvrule\tableaustep\kern\tableaustep\phantomvrule\tableaustep}%
  \hbox{\vbox{\phantomhrule\tableauside}\kern-\tableaurule}}}
\def\squares#1{\hbox{\count0=#1\noindent\loop\sqr
  \advance\count0 by-1 \ifnum\count0>0\repeat}}
\def\tableau#1{\vcenter{\offinterlineskip
  \tableaustep=\tableauside\advance\tableaustep by-\tableaurule
  \kern\normallineskip\hbox
    {\kern\normallineskip\vbox
      {\gettableau#1 0 }%
     \kern\normallineskip\kern\tableaurule}%
  \kern\normallineskip\kern\tableaurule}}
\def\gettableau#1 {\ifnum#1=0\let\next=\null\else
  \squares{#1}\let\next=\gettableau\fi\next}

\def\BZ{\mathbb Z}

\def\a{\alpha}

\def\l{\lambda}

\def\d{\delta}
\def\b{\beta}
\def\th{\theta}

\def\om{\omega}

\def\fg{\mathfrak{g}}
\def\fsl{\mathfrak{sl}}

\begin{document}


\title[The $\mathrm{SL}_3$ colored Jones polynomial of the trefoil]{
The $\mathrm{SL}_3$ colored Jones polynomial of the trefoil}
\author{Stavros Garoufalidis}
\address{School of Mathematics \\
         Georgia Institute of Technology \\
         Atlanta, GA 30332-0160, USA \newline 
         {\tt \url{http://www.math.gatech.edu/~stavros}}}
\email{stavros@math.gatech.edu}
\author{Hugh Morton}
\address{Department of Mathematics \\
        University of Liverpool \\
        Liverpool L69 3BX \\ ENGLAND \newline
        {\tt \url{http://www.liv.ac.uk/~su14}}}
\email{morton@liverpool.ac.uk}
\author{Thao Vuong}
\address{School of Mathematics \\
         Georgia Institute of Technology \\
         Atlanta, GA 30332-0160, USA \newline
         {\tt \url{http://www.math.gatech.edu/~tvuong}}}
\email{tvuong@math.gatech.edu}
\thanks{S.G. was supported in part by NSF. \\
\newline
1991 {\em Mathematics Classification.} Primary 57N10. Secondary 57M25.
}

\date{September 26, 2011}


\begin{abstract}
Rosso and Jones gave a formula for the colored Jones polynomial of a 
torus knot, colored by an irreducible representation of a simple Lie algebra.
The Rosso-Jones formula involves a plethysm function, unknown in general.
We provide an explicit formula for the second plethysm of 
an arbitrary representation of $\fsl_3$, which allows us to
give an explicit formula for the colored Jones polynomial of the trefoil,
and more generally, for $T(2,n)$ torus knots. We give two independent
proofs of our plethysm formula, one of which uses the work of 
Carini-Remmel. Our formula for the $\fsl_3$ colored Jones polynomial
of $T(2,n)$ torus knots allows us to verify the Degree Conjecture for
those knots, to efficiently
the $\fsl_3$ Witten-Reshetikhin-Turaev invariants of the Poincare sphere,
and to guess a Groebner basis for recursion ideal of
the $\fsl_3$ colored Jones polynomial of the trefoil.
\end{abstract}

\maketitle

\tableofcontents

\section{Introduction}
\lbl{sec.intro}
The initial goal of this paper was to provide a supply of explicit quantum 
invariants so as to help in formulating and testing a number of conjectures.
The most readily approachable knots in this context are the $(m,n)$ torus 
knots, particularly when $m=2$.  The aim was to give explicit details for 
the $\fsl_3$ invariants, as these are potentially the simplest case after 
the more readily available colored Jones ($\fsl_2$) invariants.

There is a general method of Rosso and Jones to determine any quantum 
invariant of a torus knot. For the invariant of the $(m,n)$ torus knot 
with quantum group module $V$ their calculations require knowledge of the 
decomposition of the module $\psi_m(V)$ into irreducible representations. 
This is a combinatorial problem depending on the quantum group and the 
choice of $V$, which does not always have a readily available explicit 
formula.

We give here an explicit formula where $m=2$ and $V$ is a general 
irreducible $\fsl_3$ module; from this we are able to give a detailed 
estimate for the extreme degrees of the resulting Laurent polynomial invariant.

Subsequently the second author reformulated some combinatorial work of 
Carini and Remmel \cite{CR} describing $\psi_2(V)$ for the irreducible 
$\fsl_N$ modules which correspond to partitions with $2$ parts.  This 
recovers the explicit formulae for $\fsl_3$, and also allows us to 
extend them to $\fsl_N$.

\section{The colored $\fsl_3$ Jones polynomial of the trefoil}
\lbl{sub.cj}

In his seminal paper \cite{Jo}, Jones introduced the Jones polynomial
of a knot $K$ in 3-space. The Jones polynomial is a Laurent polynomial in a
variable $q$
with integer coefficients, which can be generalized to an invariant
$J_{K,V}(q) \in \BZ[q^{\pm 1}]$ of a (0-framed) knot $K$ colored by a 
representation $V$ of a simple Lie algebra $\mathfrak{g}$, and normalized
to be $1$ at the unknot. The definition of $J_{K,V}(q)$ uses the 
machinery of {\em quantum groups} and may be found in \cite{Tu1,Tu2}
and also in \cite{Ja}.

Concrete formulas for the colored Jones polynomial $J_{K,V}(q)$ are hard to 
find in the case of higher rank Lie algebras, and for good reasons. 
For torus knots $T$, Jones and Rosso gave a formula for $J_{T,V}(q)$ which 
involves a plethysm map of $V$, unknown in general. Our goal is to give an 
explicit formula for the second plethysm of representations of $\fsl_3$ and 
consequently to give a formula for the $\fsl_3$
colored Jones polynomial of the trefoil.
To state our results, let $V_{n_1,n_2}$ denote the irreducible representation
of $\fsl_3$ with highest weight 
\begin{equation}
\lbl{eq.m1m2}
\l=n_1 \om_1 + n_2 \om_2
\end{equation}
where $n_1, n_2$ are non-negative integers and
$\om_1,\om_2$ are the fundamental weights of $\fsl_3$ dual to the 
simple roots $\a_1,\a_2$. In coordinates, we have
$$
\a_1=(1,-1,0), \qquad \a_2=(0, 1,-1), \qquad \om_1=\frac{1}{3}(2\a_1+\a_2),
\qquad \om_2=\frac{1}{3}(\a_1+2\a_2)
$$
The {\em quantum integer} $[n]$, 
the {\em quantum dimension} $d_{n_1,n_2}$ and the 
{\em twist parameter} $\th_{n_1,n_2}$ of $V_{n_1,n_2}$ are defined by
\begin{eqnarray}
\lbl{eq.n}
[n]&=&\frac{q^{\frac{n}{2}}-q^{-\frac{n}{2}}}{q^{\frac{1}{2}}-q^{-\frac{1}{2}}}
\\
\lbl{eq.dim}
d_{n_1,n_2}&=&\frac{[n_1+1][n_2+1][n_1+n_2+2]}{[2]}
\\
\lbl{eq.theta}
\theta_{n_1,n_2}&=&q^{\frac{1}{3}(n_1^2+n_1n_2+n_2^2)+n_1+n_2}
\end{eqnarray}
Let $T(m,n)$ denote the {\em torus knot} associated to a pair of coprime 
natural numbers $m,n$, and let $J_{T(m,n),n_1,n_2}(q)$ denote the $\fsl_3$
colored Jones polynomial of the torus knot $T(m,n)$ colored by 
$V_{n_1,n_2}$.

\begin{theorem}
\lbl{thm.1}
For all odd natural numbers $n$ we have
\begin{eqnarray*}
J_{T(2,n),n_1,n_2}(q) &=& \frac{\theta_{n_1,n_2}^{-2n}}{d_{n_1,n_2}}
\left(
\sum_{l=0}^{\min\{n_1,n_2\}}\sum_{k=0}^{n_1-l}(-1)^k d_{2n_1-2k-2l,2n_2+k-2l}
\theta_{2n_1-2k-2l,2n_2+k-2l}^{\frac{n}{2}} \right. \\ 
&   & \qquad\qquad +  \sum_{l=0}^{\min\{n_1,n_2\}}
\sum_{k=0}^{n_2-l}(-1)^k d_{2n_1+k-2l,2n_2-2k-2l}
\theta_{2n_1+k-2l,2n_2-2k-2l}^{\frac{n}{2}} \\
&   & \qquad\qquad - \left. \sum_{l=0}^{\min\{n_1,n_2\}}  d_{2n_1-2l,2n_2-2l}
\theta_{2n_1-2l,2n_2-2l}^{\frac{n}{2}} \right).
\end{eqnarray*}
\end{theorem}
Theorem \ref{thm.1} can be used to answer for several problems. 
\begin{itemize}
\item
We can verify the $\fsl_3$-Degree Conjecture of the colored Jones polynomial
for the trefoil; see \cite{GV}. Explicitly, we can compute the lowest
degree $\d^*_{T(2,n),n_1,n_2}$ and the highest degree $\d_{T(2,n),n_1,n_2}$ of
the Laurent polynomial $J_{T(2,n),n_1,n_2}(q)$ as follows
\begin{equation}
\lbl{eq.dstar}
\d^*_{T(2,n),n_1,n_2}=
\begin{cases}
-\frac{n}{2}n_1^2-\frac{n}{2}n_2^2-nn_1n_2-\frac{3n}{2}n_1
-(\frac{5n}{2}-2)n_2 &\text{if}~n_1\geq n_2\\
-\frac{n}{2}n_1^2-\frac{n}{2}n_2^2-nn_1n_2-\frac{3n}{2}n_2
-(\frac{5n}{2}-2)n_1 &\text{if}~n_1 < n_2
\end{cases}
\end{equation}
\begin{equation}
\lbl{eq.d}
\d_{T(2,n),n_1,n_2}=-(n-1)(n_1+n_2)
\end{equation}
The above formula verifies that the degree, restricted to each Kostant chamber,
is a quadratic quasi-polynomial.
\item
We can efficiently compute the Witten-Reshetikhin-Turaev invariant of the
Poincare sphere, complementing calculations of Lawrence \cite{La}.
\item
We can guess an explicit Groebner basis for the ideal of recursion relations 
of the 2-variable $q$-holonomic sequence $J_{T(2,3),n_1,n_2}(q)$; 
see \cite{GK}.
\end{itemize}

\begin{remark}
\lbl{rem.thm1}
An alternative formula for the $\fsl_3$ colored Jones polynomial of $T(2,3)$
is given by Lawrence in \cite{La}. Lawrence's formula is derived from
the theory of Quantum Groups, and cannot generalize to the case of $T(2,n)$
torus knots. In contrast, the plethysm formula of Theorem
\ref{thm.2} below can be generalized to a formula for $\psi_m(V_{\l})$
which allows for an efficient formula of the $\fsl_3$ colored Jones polynomial
of all torus knots. Additional generalizations are possible for all 
simple Lie algebras; see \cite{GV}.
\end{remark}

\begin{remark}
\lbl{rem.katlas}
Theorem \ref{thm.1} gives an efficient computation of the $\fsl_3$ colored 
Jones polynomial of the $3_1,5_1,7_1$ and $9_1$ knots in the Rolfsen notation.
In low weights, our answer agrees with the independent computation 
given by the entirely different methods of the {\tt KnotAtlas}; see
\cite{B-N}. This is a consistency check which simultaneously validates
the formulas of Theorem \ref{thm.1} and the data of the {\tt KnotAtlas}. 
\end{remark}

\subsection{An $\fsl_3$ plethysm formula}
\lbl{sub.plethysm}

As mentioned above, Theorem \ref{thm.1} follows from the Rosso-Jones
formula for the colored Jones polynomial of torus knots and the
following plethysm computation. Let $\psi_m$ denote the $m$-plethysm
operation.

\begin{theorem}
\lbl{thm.2}
For $\l$ as in Equation \eqref{eq.m1m2} we have
\begin{eqnarray*}
\psi_2(V_{\l})&=&
\sum_{l=0}^{\min\{n_1,n_2\}}\sum_{k=0}^{n_1-l}(-1)^k V_{2 \l-k \a_1-2 l(\a_1+\a_2)} \\
& & + 
\sum_{l=0}^{\min\{n_1,n_2\}}\sum_{k=0}^{n_2-l}(-1)^k 
V_{2 \l-k \a_2-2 l(\a_1+\a_2)} \\
& & - 
\sum_{l=0}^{\min\{n_1,n_2\}} V_{2 \l-2 l(\a_1+\a_2)}
\end{eqnarray*}
\end{theorem}

\section{The Rosso-Jones formula}
\lbl{sec.jr}

The polynomial invariant $J_{K,V}(q)$ of a knot $K$ colored by the representation
$V$ of a simple Lie algebra is difficult to compute from its Quantum Group
definition 
even when $K=4_1$ and $\fg=\fsl_3$. Although it is a finite multi-dimensional
sum, a practical computation seems out of reach. Fortunately, there is
a class of knots whose quantum group invariant has a simple enough formula
that allows us to extract its $q$-degree. This is the class of 
{\em torus knots} $T(m,n)$ where $m,n$ are coprime natural numbers. The simple
formula is due to  Rosso and Jones, and also studied by the second named
author, \cite{JR,Mo}.
Let $d_{\l}$ denote the {\em quantum dimension} of the representation $V_{\l}$
and $\th_{\l}$ is the eigenvalue of the {\em twist} operator on  
the representation $V_{\l}$. $d_{\l}$ and $\th_{\l}$ are given by

\begin{eqnarray}
\lbl{eq.dl}
d_{\l}&=&\prod_{\a>0} \frac{[(\l+\rho,\a)]}{[(\rho,\a)]}
\\
\lbl{eq.thl}
\th_{\l}&=&q^{\frac{1}{2}(\l,\l+2\rho)}
\end{eqnarray}
where $\a$ belongs to the set of positive roots,  
$\rho=\frac{1}{2} \sum_{\a>0} \a$
is half the sum of positive roots and $(\cdot,\cdot)$ denotes the 
$\mathfrak{g}$ invariant inner product on the dual of the Cartan algebra 
(normalized so that the longest root has length $\sqrt{2}$). 
When $\mathfrak{g}=\fsl_3$ and 
$\l$ is given by \eqref{eq.m1m2}, then the quantum dimension and the
twist parameter coincide with \eqref{eq.dim} and \eqref{eq.theta}.
For a natural number $m$, consider the $m$-{\em Adams operation} $\psi_m$
on representations. It is given by (see \cite{FH,Mc})
\begin{equation}
\lbl{eq.Sla}
\psi_m(V_{\l})=\sum_{\mu \in S_{\l,m}} c^{\mu}_{\l,m} V_{\mu}
\end{equation}
where $c^{\mu}_{\l,m}$ are non-zero integers.  
The Rosso-Jones formula is the following (see \cite{JR}):

\begin{equation}
\lbl{eq.jr}
J_{T(m,n),\l}(q)=\frac{\th_{\l}^{-mn}}{d_{\l}}
\sum_{\mu \in S_{\l,m}} c^{\mu}_{\l,m} d_{\mu} \th_{\mu}^{\frac{n}{m}}
\end{equation}
For related discussion, see also \cite{MM}.

\section{Schur functions in $\fsl_3$}
\lbl{sec.sl3}

\subsection{A review of Schur functions}
\lbl{sub.sreview}

Let us recall some well-known properties of Schur functions and their
relation to the character of irreducible representations of $\fsl_N$,
that can be found in \cite{Mc,FH}. For a partition $\l$ with parts 
$\l_1\ge\l_2\ge \ldots\ge\l_k\ge0$, let $s_{\l_1,\ldots,\l_k}(x_1,\ldots,x_N)$
denote the corresponding Schur function. A partition $\l=(\l_1,\,\dots,\l_k)$
will be depicted as an arrangement of boxes as follows (for $\l=(4,2,1)$):
$$
\psdraw{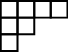}{0.8in}
$$ 
If $\om_i$ denote the fundamental weights of $\fsl_N$ and $n_i$ are 
nonnegative integers for $i=1,\dots,N-1$, and 
$\l=(\sum_{i=1}^{N-1} n_i, \sum_{i=2}^{N-1} n_i, \dots, \sum_{i=N-1}^{N-1} n_i)$
then 
\begin{equation}
\lbl{eq.sVL}
\text{character}(V_{\sum_{i=1}^{N-1} n_i \om_i})=
s_{\l}(x_1,\ldots,x_N)
\end{equation}
For $\l=(4,2,1)$ we then have $(n_1,n_2,n_3)=(2,1,1)$.

The plethysm operation $\psi_m$ is defined by 
$$
\psi_m(s_\l(x_1,\ldots,x_N))=s_\l(x_1^m,\ldots,x_N^m)
$$
Note that $s_{1}=x_1+\cdots+x_N$ and $\psi_2(s_{1})=s_{2}-s_{1,1}$.

In $\fsl_N$ the irreducible modules correspond to partitions $\l$ with at 
most $N$ parts. 
The decomposition of $\psi_m(V_\l)$ into irreducibles needed for the 
invariant of the $(m,n)$ torus knot is given by the corresponding expansion 
of the symmetric function $\psi_m(s_\l)$ as a linear combination of Schur 
functions.

When $N=3$ the Schur function $s_\l$ vanishes where $\l$ has more than $3$ 
parts, and satisfies $s_{a,b,c}=s_{a+1,b+1,c+1}$. Then $s_{a,b,c}=s_{a-c,b-c}$,  
so we need only 
consider partitions with at most $2$ parts.  All the same, it will be 
convenient to use $3$ parts in what follows.

\subsection{A reformulation of Theorem \ref{thm.2}}

The goal of this section is to give a formula for $\psi_2(s_{m_1,m_2})$ as a 
linear combination of Schur functions, assuming that $N=3$. 

\begin{definition}
\lbl{def.Dm1m2}
For $m_1\ge m_2\ge0$, let  $D(m_1,m_2)\subset {\mathbb{N}}^3$ denote the
set of tuples $(a,b,c)$ that satisfy
\begin{itemize}
\item
$a+b+c=2m_1+2m_2,\  2m_1\ge a\ge b\ge c\ge 0,\  a\ge 2m_2\ge c$
\item
if $b \geq 2 m_2$ then $c\equiv 0 \bmod 2$
\item
if $b \leq 2 m_2$ then $a\equiv 0 \bmod 2$
\end{itemize}
\end{definition}

\begin{theorem}
\lbl{psi2}
In $\fsl_3$ for all $m_1 > m_2$ we have:
$$
\psi_2(s_{m_1,m_2})=\sum_{(a,b,c)\in D(m_1,m_2)} (-1)^b s_{a,b,c}
$$
\end{theorem}

It is interesting to note that the coefficient of every Schur function in 
the expansion of $\psi_2(s_{m_1,m_2})$ is $0,\pm1$. The same feature proves to 
be the case for $\psi_2(s_{m_1,m_2})$ in the general case of $\fsl_N$, noted in 
Subsection \ref{sub.thm5}.

\subsection{Theorem \ref{psi2} implies Theorem \ref{thm.2}}

Since $V_{n_1 \om_1 + n_2 \om_2}^*=V_{n_2 \om_1 + n_1 \om_2}$, and $J_{K,V^*}(q)=J_{K,V}(1/q)$,
it suffices to prove Theorem \ref{thm.2} when $n_1 > n_2$.
Equation \eqref{eq.sVL} for $N=3$ implies that
$$
\text{character}(V_{n_1 \om_1 + n_2 \om_2})=s_{n_1+n_2,n_2}(x_1,x_2,x_3)
$$
Fix nonnegative integers $n_1$ and $n_2$ and set
$(m_1,m_2)=(n_1+n_2,n_2)$ in Theorem \ref{psi2}.

We can parametrise a tuple $(a,b,c) \in D(m_1,m_2)$ that satisfies
$b\ge 2m_2$ by setting 
$b=2m_2+k,~ c=2l$, to get $a=2m_1-k-2l$, satisfying the inequalities 
$k,l\ge 0$, $k\le m_1-m_2-l$, $l\le m_2, m_1-m_2$. Likewise,
we can parametrize a tuple $(a,b,c) \in D(m_1,m_2)$ that satisfies
$b\le 2m_2$ by setting $b=2m_2-k, a=2m_1-2l$ to get $c=2l+k$, satisfying 
$k,l\ge 0$, $k\le m_2-l$, $l\le m_2, m_1-m_2$. Thus Theorem \ref{psi2}
implies the formula of Theorem \ref{thm.2}.

\subsection{A reformulation of Theorem \ref{psi2}}

To establish Theorem \ref{psi2} we first prove Theorem \ref{psi2s}.

\begin{theorem}
\lbl{psi2s}
For $m_1>m_2$ we have
\begin{eqnarray*}
\left(\sum_{(a,b,c)\in\substack { D(m_1,m_2)}} (-1)^b s_{a,b,c}\right) 
\psi_2( s_{1})&=& \sum_{(a',b',c')\in\substack { D(m_1+1,m_2)}} (-1)^{b'} s_{a',b',c'} 
\\ & &+ \sum_{(a',b',c')\in\substack{ D(m_1,m_2+1)}} (-1)^{b'} s_{a',b',c'}\\
& & + \sum_{\substack{(a',b',c')\in D(m_1-1,m_2-1) \\ m_2>0}} (-1)^{b' }s_{a',b',c'} .
\end{eqnarray*}
\end{theorem}
 
In the proof of Theorem \ref{psi2} we will need the following special 
cases of the {\em Littlewood-Richardson rule}
adapted to $\fsl_3$, bearing in mind that Schur functions for partitions 
with more than $3$ parts are $0$ in this case; see \cite{Mc}. In the next
lemma and below, we will use the convention that $s_{a_1,a_2,a_3}=0$ unless
$a_1 \geq a_2 \geq a_3$. Furthermore, the notation $s_{a,b,c}|_{a>b}$ (resp.
$s_{a,b,c}|_{a=b}$) means $s_{a,b,c}$ when $a>b$ (resp. $a=b$)
and zero otherwise.

\begin{lemma}
\lbl{lem.LR}
In $\fsl_3$ we have
\begin{eqnarray*}
s_{a,b,c}s_2 &=&  s_{a+2,b,c}+s_{a,b+2,c} +s_{a,b,c+2} +s_{a+1,b+1,c}|_{a>b}
+s_{a+1,b,c+1} +s_{a,b+1,c+1}|_{b>c} 
\\
s_{a,b,c}s_{1,1} &=&  s_{a+1,b+1,c}+s_{a+1,b,c+1} +s_{a,b+1,c+1}
\\
s_{m_1,m_2}s_1 &=& s_{m_1+1,m_2}+s_{m_1,m_2+1} +s_{m_1,m_2,1}
\\
\end{eqnarray*}
\end{lemma}
%

\begin{corollary}
\lbl{cor.LR}
For $a\ge b\ge c\ge0$ we have 
\begin{eqnarray*}
s_{a,b,c}(s_{2}-s_{ 1,1})& =& 
s_{ a+2,b,c}+s_{ a,b+2,c} +s_{ a,b,c+2} -s_{ a+1,b+1,c}|_{a=b} -s_{ a,b+1,c+1}|_{b=c}
\end{eqnarray*}
\end{corollary}
        
\begin{corollary}
\lbl{lem.psi2}
 
Since $\psi_2$ is a ring homomorphism, and $\psi_2(s_1)=s_2-s_{1,1}$, we have 
\begin{eqnarray*}
\psi_2(s_{m_1,m_2})(s_2-s_{1,1}) &=& 
\psi_2(s_{m_1,m_2})\psi_2(s_1) =\psi_2(s_{m_1,m_2}s_1) \\
& =&\left\{\begin{array}{lc}
\psi_2(s_{m_1+1,m_2})+\psi_2(s_{m_1,m_2+1}) +\psi_2(s_{m_1,m_2,1}) 
&\mbox{ if }m_1>m_2>0,\\
\psi_2(s_{m_1+1,m_2})+\psi_2(s_{m_1,m_2+1}) &\mbox{ if }m_1>m_2=0.
\end{array}
\right.
 \end{eqnarray*}
\end{corollary}

\subsection{Theorem \ref{psi2s} implies Theorem \ref{psi2}}

We deduce Theorem \ref{psi2} from Theorem \ref{psi2s} by induction on $m_2$. 

When $m_2=0$ we have $(a,b,c)\in D(m_1,0) ~\text{iff}~ c=0,
~ a+b=2m_1,~ a\ge b\ge0$. It is known   (for example, \cite[Eqn.2.30]{CGR}) 
that
$$
\psi_2(s_{m})=\sum_{k=0}^m (-1)^ks_{2m-k,k}.
$$ 
This establishes Theorem \ref{psi2} for $m_2=0$. 

Theorem \ref{psi2s}  gives
$$
\psi_2(s_{m_1,m_2}) \psi_2( s_{1})=\psi_2(s_{m_1+1,m_2}) 
+ \sum_{(a',b',c')\in\substack D(m_1,m_2+1)} (-1)^{b'} s_{a',b',c'}+\psi_2(s_{m_1-1,m_2-1})
$$ 
by induction on $m_2$

Corollary \ref{lem.psi2} then shows that
$$
\psi_2(s_{m_1,m_2+1})=\sum_{(a',b',c')\in D(m_1,m_2+1)} (-1)^{b'} s_{a',b',c'},
$$ 
which completes the induction step.

\subsection{Proof of Theorem \ref{psi2s}}
To prove theorem \ref{psi2s} we sum both sides of the equation in 
Corollary \ref{cor.LR} over $(a,b,c)\in D(m_1,m_2)$, using the following lemma.

\begin{lemma}
\lbl{detail} 
Suppose that $m_1>m_2\ge 0$. 
Then
{\small
\begin{eqnarray}
\lbl{eq.d1}
\sum_{(a,b,c) \in D(m_1,m_2)}  (-1)^b s_{a+2,b,c} 
&=& \sum_{\substack{(a',b',c') \in D(m_1+1,m_2) \\ a'\ne b', a'\ne 2m_2}} (-1)^{b'}s_{a',b',c'}\\[3mm]
\lbl{eq.d2}
\sum_{(a,b,c) \in D(m_1,m_2)} (-1)^b s_{a,b+2,c}
&=& \sum_{\substack{(a',b',c') \in D(m_1,m_2+1) \\ b'\ne c', c'\ne 2m_2+2}}(-1)^{b'} s_{a',b',c'}
 + \sum_{\substack{(a',b',c') \in D(m_1+1,m_2) \\ a'=2m_2, b'\ne c'}}(-1)^{b'} s_{a',b',c'}\\[3mm]
\lbl{eq.d3}
\sum_{(a,b,c) \in D(m_1,m_2)} (-1)^b s_{a,b,c+2}
&=& \sum_{\substack{(a',b',c') \in D(m_1-1,m_2-1) \\ m_2>0}}(-1)^{b'} s_{a',b',c'}  
+ \sum_{\substack{(a',b',c') \in D(m_1,m_2+1) \\ c'=2m_2+2}} (-1)^{b'} s_{a',b',c'}\\[3mm]
\lbl{eq.d4}
\sum_{\substack{(a,b,c) \in D(m_1,m_2)\\ a = b}} (-1)^{b+1} s_{a+1,b+1,c}
&=& \sum_{\substack{(a',b',c') \in D(m_1+1,m_2) \\ a'= b', a'\ne 2m_2, b' \ne c'}}(-1)^{b'} s_{a',b',c'} \\[3mm]
\lbl{eq.d5}
\sum_{\substack{(a,b,c) \in D(m_1,m_2) \\b = c}} (-1)^{b+1} s_{a,b+1,c+1}
&=& \sum_{\substack{(a',b',c') \in D(m_1,m_2+1) \\ b'= c', c'\ne 2m_2+2}}(-1)^{b'} s_{a',b',c'}
+ \sum_{\substack{(a',b',c') \in D(m_1+1,m_2) \\ a'=2m_2, b'=c'}} (-1)^{b'} s_{a',b',c'}
\end{eqnarray}
}
\end{lemma}
 
The total sum of the left hand sides of the equations in Lemma \ref{detail} 
is then the left hand side of the equation in theorem \ref{psi2s}, while 
the terms on the right hand sides make up the right hand side of Theorem 
\ref{psi2s}.

\subsection{Proof of Lemma \ref{detail}}
For each of the five equations we provide a bijective transformation 
carrying $(a,b,c)\in D(m_1,m_2)$ with the restrictions shown to $(a',b',c')$ 
satisfying the conditions on the right hand sides. 
 
We  make repeated use of  the parity rules to ensure that inequalities 
force a difference of at least $2$.  With the exception of a couple of 
less obvious cases we omit proofs that the individual parity rules for 
$(a',b',c')$ are satisfied, as they generally follow readily from those 
for $(a,b,c)$ and vice versa. Equally the sum $a'+b'+c'$  is always 
obviously correct.

\begin{proof}
For Equation \eqref{eq.d1}, put $a'=a+2,b'=b,c'=c$.
Let $(a,b,c)\in D(m_1,m_2)$. Then $2m_2+2\ge a'> b'\ge c'\ge0$, 
and $a'>2m_2\ge c'$. Then $(a',b',c')\in D(m_1+1,m_2)$, with 
$a'\ne b'$ and $a' \ne 2m_2$.
 
Conversely suppose that $(a',b',c')\in D(m_1+1,m_2)$, with $a'> b'$ and 
$a' > 2m_2$. By the parity rules, if $b'\le 2m_2$ then $a'\equiv 0 \bmod 2$, 
so $a'\ge 2m_2+2\ge b'+2$.  If $b'>2m_2$ then $a'\equiv b' \bmod 2$, 
so $a'\ge b'+2>2m_2+2$. In any case $2m_1\ge a'-2\ge b'\ge c'\ge 0$, 
and $a'-2\ge 2m_2\ge c'$. Then $(a,b,c)\in D(m_1,m_2)$. This proves Equation
\eqref{eq.d1}.\\[0.5mm]

 For Equation \eqref{eq.d2}, put $a'=a,b'=b+2,c'=c$. Let 
$(a,b,c)\in D(m_1,m_2)$ with $a\ge b+2$. If $a=2m_2$ then 
$2m_1+2\ge a'\ge b'> c'\ge 0$ and $a'\ge 2m_2\ge c'$, Then 
$(a',b',c')\in D(m_1+1,m_2)$, with  $a' =2m_2, b'>c'$. Otherwise $a>2m_2$. 
If $b\ge 2m_2$ then  $a\ge b+2\ge 2m_2+2$, while if $b<2m_2$ then 
$a\equiv 0 \bmod 2$ by the parity rules, so that $a\ge 2m_2+2$. Hence 
$2m_1\ge a'\ge b'> c'\ge 0$ and $a'\ge 2m_2+2\ge c'$. In this case we 
check the parity rules explicitly. Here $b'\ge 2m_2+2\implies b\ge 
2m_2\implies c'\equiv c\equiv 0 \bmod 2$ and $b'\le 2m_2+2\implies 
b\le 2m_2\implies a'\equiv a\equiv 0 \bmod 2$. So  $(a',b',c')\in 
D(m_1,m_2+1)$ with $b'>c'$ and $c'<2m_2+2$.

 Conversely suppose that $(a',b',c')\in D(m_1,m_2+1)$ with $b'>c'$ and 
$c'<2m_2+2$. If $b'\ge 2m_2+2$ then $c'\equiv 0\bmod 2$ so $c'\le 2m_2\le 
b'-2$ and if $b'<2m_2+2$ then $b'\equiv c' \bmod 2$ and $c'\le b'-2<2m_2$. 
Hence $2m_1\ge a'\ge b'-2\ge c'\ge 0$ and $a'>2m_2$. A parity check as 
above shows that then $(a,b,c)\in D(m_1,m_2)$ with $a=a'\ge b'=b+2$ and 
$a>2m_2$.
 
  Finally suppose that $(a',b',c')\in D(m_1+1,m_2)$, with  $a' =2m_2, b'>c'$.
  Then $b'\equiv c'\bmod 2$ so $b'-2\ge c'$, and $a'=2m_2\ge c'$   again 
giving $(a,b,c)\in D(m_1,m_2)$ with $a=2m_2\ge b+2$. This proves Equation 
\eqref{eq.d2}.\\[0.5mm]

 For Equation \eqref{eq.d3}, put $a'=a,b'=b,c'=c+2$ when $c=2m_2$, and 
$a'=a-2,b'=b-2,c'=c$ otherwise. In either case $s_{a,b,c+2}=s_{a',b',c'}$ since 
we are working in $\fsl_3$. Let $(a,b,c)\in D(m_1,m_2)$ with $b\ge c+2$. 
If $c=2m_2$ then $2m_1\ge a'\ge b'\ge 2m_2+2=c'\ge 0$, and  $(a',b',c')\in 
D(m_1,m_2+1)$ with  $c'=2m_2+2$. Otherwise $c<2m_2\ne 0$. If $b\le2m_2$ then 
$c\le 2m_2-2$. If $b>2m_2$ then $c\equiv 0\bmod 2$ by the parity rules, 
giving again $c\le 2m_2-2$. Then $2m_1-2\ge a-2\ge b-2\ge c\ge 0$ and 
$a-2\ge 2m_2-2\ge c$. So $(a',b',c')\in D(m_1-1,m_2-1)$.
 
  Conversely let $(a',b',c')\in D(m_1-1,m_2-1)$, with $m_2\ne 0$. Then 
$2m_1\ge a'+2\ge b'+2\ge c'\ge 0$ and $a'\ge 2m_2-2\ge c'$, so 
$a'+2\ge 2m_2>c'$. Hence $(a,b,c)\in D(m_1,m_2)$ with $c\ne 2m_2$.
  
  Finally, suppose that $(a',b',c')\in D(m_1,m_2+1)$ with  $c'=2m_2+2$. 
Then $2m_1\ge a'\ge b'\ge 2m_2=c'-2\ge 0$ so that $(a',b',c'-2)=(a,b,c)
\in D(m_1,m_2)$ with $c=2m_2$. This proves Equation \eqref{eq.d3}.\\[0.5mm]

 For Equation \eqref{eq.d4}, put $a'=a+1,b'=b+1,c'=c$. Let $(a,b,c)\in 
D(m_1,m_2)$ with $a=b$.  Then $2m_1+2\ge a'\ge b'\ge c'\ge 0$ and $a'>a\ge 
2m_2\ge c'\ge 0 $. Since $b'=a'>2m_2$ and $c'\equiv 0\bmod 2$ the parity 
rules are satisfied, and $(a',b',c')\in D(m_1+1,m_2)$ with $a'=b', a'>2m_2, 
b'\ne c'$. 
 
 Conversely let $(a',b',c')\in D(m_1+1,m_2)$ with $a'=b', a'>2m_2, b'\ne c'$. 
Now $2a'\le a'+b'+c'=2m_1+2m_2+2\le 4m_1$, since $m_2<m_1$. Then $2m_1>a'-1
\ge b'-1\ge c'\ge 0$ and $a'-1\ge 2m_2\ge c'$. Hence $(a,b,c)\in D(m_1,m_2)$ 
with $a=b$. This proves Equation \eqref{eq.d3}.\\[0.5mm]

 For Equation \eqref{eq.d5}, put $a'=a,b'=b+1,c'=c+1$. Let $(a,b,c)\in 
D(m_1,m_2)$ with $a>b=c$. If $a=2m_2$ then $2m_1+2>a'\ge b'\ge c'\ge 0$ and 
$a'=2m_2\ge c'$. Hence $(a',b',c')\in D(m_1+1,m_2)$ with $a'= 2m_2, b'= c'$. 
Otherwise $a>2m_2$,  and $a'=a\ge 2m_2+2$, since $b=c$, while $2m_2+2\ge 
c+2>c'$. We have also $2m_1\ge a'\ge b' \ge c'\ge 0$. Hence $(a',b',c')
\in D(m_1,m_2+1)$ with $b'=c', c'\ne 2m_2+2$. 
 
 Conversely suppose that $(a',b',c')\in D(m_1,m_2+1)$ with $b'=c', c'< 
2m_2+2$. Now $a'+2c'=2m_1+2m_2+2$ and $a'\le 2m_1$, so $c'>0$. Hence $2m_1
\ge a'>b'-1\ge c'-1\ge 0$ and $a'> 2m_2\ge c'-1$. Then $(a,b,c)\in D(m_1,m_2)$ 
with $a>b=c$ and $a>2m_2$. 
 
 Finally if $(a',b',c')\in D(m_1+1,m_2)$ with $a'= 2m_2, b'= c'$ then 
$b'=c'=m_1+1>0$ and $(a,b,c)\in D(m_1,m_2)$ with $a=2m_2>b=c$.
\end{proof}

\section{A proof of Theorem \ref{psi2} using Carini-Remmel's work}
\lbl{sec.CR}

\subsection{A review of Theorem 5 of \cite{CR}}
\lbl{sub.thm5}

In this section we give an alternative proof of Theorem \ref{psi2}
using the work \cite{CR} of Carini and Remmel. In Theorem 5 of loc.cit.,
Carini and Remmel give 
the expansion of the plethysm $\psi_2(s_{a,b})$ for the Schur function of a 
2-row partition of $n=a+b$ in terms of Schur functions $s_\lambda$, where 
$\lambda$ runs through partitions of $2n$ with at most $4$ parts.
In this expansion each $s_\lambda$ has coefficient $0,\pm1$, depending on the 
parities of the parts of $\lambda$ and some linear inequalities.

In their paper they use the opposite convention to Macdonald, so that 
they take $0\le a\le b$ for the given partition of $n=a+b$ and 
$0\le \l_1\le\l_2\le\l_3\le\l_4$ for the parts of the partition $\l$ 
of $2n$. They also use the more common combinatorial notation $p_2$ 
rather than $\psi_2$.

Theorem 5 of \cite{CR} can be readily restated as follows, by grouping 
separately the partitions $\l$ of $2a+2b$ with $\l_1+\l_3\ge 2a$ and 
those with $ \l_1+\l_3< 2a$ in the expansion of $\psi_2(s_{a,b})$:
 
\begin{itemize}
\item
When $\l_1+\l_3\ge 2a$, $\l_1+\l_2$ is even and $\l_1+\l_2\le 2a$,  
the Schur function $s_\l$ has coefficient $(-1)^{\l_2+\l_3}$.
\item
When $\l_1+\l_3< 2a$, $\l_2+\l_3$ is even,  $2a\le\l_2+\l_3$ and 
$2a\le\l_1+\l_4$, the Schur function $s_\l$ has coefficient $(-1)^{\l_1+\l_2}$.
\item
All other $s_\l$ have coefficient $0$.
\end{itemize}
 
The first of these cases corresponds to the partitions in (ii) and some of 
(i) in \cite[Thm.5]{CR}, while the second corresponds to the partitions 
in (iii) and the remaining partitions in (i).

\subsection{Reformulation of Carini and Remmel's expansion of 
$\psi_2(s_{m_1,m_2})$}

Theorem 5 of \cite{CR} gives rise to an expansion of 
$\psi_2(s_{m_1,m_2}), m_1\ge m_2$,   in Schur functions of $x_1,\ldots,x_N$ 
which is valid for all $N$.

We can reformulate this further by specifying the support set for the 
partitions which appear in the expansion in terms of linear inequalities 
and some parity rules, so that Theorem \ref{psi2}, the case where $N=3$, 
is an immediate corollary.

Using Macdonald's ordering, we take $m_1$ in place of $b$ and $m_2$ in place 
of $a$ from \cite{CR}, and write $(\l_4,\l_3,\l_2,\l_1)=(a,b,c,d)=\l$. 
\begin{definition} For $m_1,~m_2\in\mathbb{N}$, let  $A(m_1,m_2)\subset \mathbb{N}^4$ denote the set of tuples $(a,b,c,d)$ that satisfy
\begin{itemize}
\item $a+b+c+d=2m_1+2m_2,\  a \ge b \ge c \ge d \ge 0 ,\  2m_1 \ge a+d \ge 2m_2 \ge c+d$
\item if $ b+d \ge 2m_2$ then $ c \equiv d \bmod 2$
\item if $ b+d \le 2m_2$ then $ a \equiv d \bmod 2$
\end{itemize}
\end{definition}

\begin{theorem}
\lbl{psi2N}
Let $m_1\ge m_2\ge 0$.  Then
$$\psi_2(s_{m_1,m_2})=\sum_{(a,b,c,d)\in A(m_1,m_2)} (-1)^{b+d}s_{a,b,c,d}.$$
\end{theorem}

Theorem \ref{psi2} is an immediate corollary, since Schur functions for 
partitions with more than $3$ rows are $0$ in $\fsl_3$, and the support 
set $A(m_1,m_2)$ becomes $D(m_1,m_2)$ when $d=0$. 

We can see readily that theorem \ref{psi2N} follows from Theorem 5 of \cite{CR} as 
rearranged above.  

Firstly, for $\l\in A(m_1,m_2)$ with $b+d\ge 2m_2$ we have $c+d$ even, by 
the parity rule, and $c+d\le 2m_2$, while the coefficient of $s_\l$ is 
$(-1)^{b+d}=(-1)^{b+c}$. This agrees with  the first group of partitions above. 
The condition $2m_1\ge a+d$ does not impose any extra restriction on this 
group, since it is equivalent to $b+c\ge 2m_2$.   

For $\l\in A(m_1,m_2)$ with $b+d\le 2m_2$ we have $a+d$ even, and hence 
$b+c$ even, by the parity rule. In addition we have $2m_2\le b+c$  since 
$2m_1\ge a+d$, and $2m_2\le a+d$. Again this agrees with the second group 
of partitions above, and the coefficient of $s_\l$ is $(-1)^{b+d}=(-1)^{c+d}$ 
as required there.

\subsection{Parametrisation}
\lbl{sub.parametrisation}

Theorem \ref{psi2N} can be used to give a parametrisation of these two 
sets of Schur 
functions with non-zero coefficient, each in terms of $3$ integer 
parameters satisfying some linear inequalities. These in turn give a 
parametric formula for $\psi_2(s_{m_1,m_2})$, with a reduction in the case of 
$\fsl_3$ 
to the formulae of Theorem \ref{thm.2}. 

\subsubsection{The first group of Schur functions}
\lbl{subsub1}

 Parametrise $\{A(m_1,m_2):b+d\ge 2m_2\}$ by setting $b+d=2m_2+k, k\ge0$. 
Write $c=d+2l, l\ge0$ to get $c\equiv d\bmod2$. The condition $c+d\le 2m_2$ 
is equivalent to $d+l\le m_2$. This ensures that $c\le b$. Then 
$a=2m_1-k-2l-d$, which satisfies $2m_1\ge a+d$. To ensure that $a\ge b$ we   
impose the condition $a+d=2m_1-k-2l\ge b+d=2m_2+k$ to finish with parameters 
$k,l,d\ge 0, d+l\le m_2, k+l\le m_1-m_2$. 
  
  The contribution of the partitions $\l$ with $b+d\ge 2m_2$ is then
$$
\sum (-1)^k s_\l,  \mbox{ where } 
\l=(2m_1-k-2l-d,2m_2+k-d,2l+d,d)
$$ 
and $k,l,d$ are integer parameters with 
$ k,l,d \ge0,  d+l\le m_2, k+l\le m_1-m_2$.
 
\subsubsection{The second group of Schur functions} 
\lbl{subsub2}
  
   Parametrise $\{A(m_1,m_2):b+d\le 2m_2\}$ by setting $b+d=2m_2-k, k\ge0$. 
 Write $a+d=2m_1-2l,l\ge 0$ to get $a\equiv d\bmod 2$ and $2m_1\ge  a+d$.
 Then $b+c=2m_2+2l$, so $c\ge d$. 
 The condition $2m_2\le a+d$ is equivalent to $l\le m_1-m_2$. This ensures 
that $b\le a$.  
 
 Now $b=2m_2-k-d$ so $c=2l+k+d$ so $c\le b$ is equivalent to $l+k+d\le m_2$.

 
  The contribution of the partitions $\l$ with $b+d\le 2m_2$ is
$$
\sum (-1)^k s_\l,  \mbox{ where } 
\l=(2m_1-2l-d,2m_2-k-d,2l+k+d,d)
$$ 
and $k,l,d$ are integer parameters with 
$ k,l,d \ge0,  l+k+d\le m_2, l\le m_1-m_2$.

\subsection{Reduction to the case of $\fsl_3$.}
\lbl{sub.reduction}

In the special case of $\fsl_3$ we have $d=0$, and we get two double sums 
of 3-row Schur functions, one for partitions with $b\ge 2m_2$, and one for 
those with $b<2m_2$, to avoid double counting those with $b=2m_2$.
Since we are working in $\fsl_3$ this can be reduced further to sums over 
$2$-row partitions, since $s_{a,b,c}=s_{a-c,b-c}$ 

Explicitly we have from the first group of partitions the sum
$$
\sum (-1)^k s_{2m_1-4l-k, 2m_2-2l+k}
$$ 
taken over $k,l\ge0, l\le m_2,k+l\le m_1-m_2$. The second group yields 
$$ 
\sum (-1)^k s_{2m_1-4l-k, 2m_2-2l-2k}
$$ 
taken over $l\ge 0, k>0, k+l\le m_2, l\le m_1-m_2$. This gives a second
proof of Theorem \ref{thm.2}.    
It may be preferable all the same to retain the 3-row format when estimating 
the 
effects of twists in $\fsl_3$ as then all the partitions have $2m_1+2m_2$ 
cells and thus their twist factors depend only on the total content of 
the partition.  

%

\section{Sample computations}
\lbl{sec.compute}

In this section we give some sample computations of Theorems \ref{thm.1}
and \ref{thm.2}. Theorem \ref{thm.1} implies that:

{\small
\begin{math}
J_{T(2,3),5,7}(1/q)=q^{24}+q^{30}+q^{32}-q^{35}+q^{36}+2 q^{38}-q^{39}-q^{41}+q^{42}-q^{43}+2 q^{44}-q^{45}-2 q^{47}+q^{48}-q^{49}+2 q^{50}-2 q^{51}+q^{52}-2 q^{53}-2 q^{55}+3 q^{56}
-2 q^{57}+2 q^{58}-2 q^{59}-q^{60}-q^{61}+2 q^{62}-4 q^{63}+3 q^{64}+q^{66}-q^{67}+q^{68}-3 q^{69}+3 q^{70}-2 q^{71}+3 q^{72}+q^{73}-q^{74}-q^{75}-2 q^{77}+2 q^{78}+q^{79}
+2 q^{80}-2 q^{82}-q^{83}+q^{85}+2 q^{86}-3 q^{88}+q^{89}-2 q^{90}-q^{92}+2 q^{93}+q^{94}+2 q^{95}-3 q^{96}+q^{97}-2 q^{98}+q^{99}+q^{100}+2 q^{101}-2 q^{102}+3 q^{103}
-5 q^{104}-q^{106}+3 q^{107}+2 q^{108}+4 q^{109}-4 q^{110}+3 q^{111}-3 q^{112}-2 q^{113}+q^{114}+q^{115}-q^{116}+5 q^{117}-5 q^{118}-2 q^{119}-2 q^{121}+2 q^{122}
+5 q^{123}-2 q^{124}+q^{125}-q^{126}-4 q^{127}-q^{129}-q^{130}+4 q^{131}-q^{132}-2 q^{133}+2 q^{134}-q^{135}+q^{136}+q^{137}-2 q^{138}+2 q^{139}+3 q^{140}-3 q^{141}
+2 q^{142}-2 q^{143}-4 q^{144}+2 q^{145}+6 q^{148}-2 q^{149}-q^{151}-6 q^{152}+3 q^{153}+4 q^{154}-q^{155}+3 q^{156}-4 q^{157}-4 q^{158}+3 q^{159}-3 q^{160}+2 q^{161}
+4 q^{162}-3 q^{163}+4 q^{164}-2 q^{165}-4 q^{166}+5 q^{167}+2 q^{170}-6 q^{171}+2 q^{172}+3 q^{173}-4 q^{174}+q^{175}+q^{176}-3 q^{177}+5 q^{178}-2 q^{179}-2 q^{180}
+4 q^{181}-2 q^{183}-q^{184}-6 q^{185}+3 q^{186}+2 q^{187}+2 q^{189}+q^{190}-5 q^{191}+2 q^{192}-q^{193}-q^{194}+5 q^{195}+2 q^{196}-q^{197}-q^{198}-5 q^{199}+3 q^{201}
-2 q^{202}+q^{203}+3 q^{204}-2 q^{205}+q^{206}-5 q^{208}+4 q^{209}+2 q^{210}-3 q^{213}-3 q^{214}+4 q^{215}-2 q^{216}+2 q^{217}+3 q^{218}-2 q^{219}-4 q^{222}+5 q^{223}
+2 q^{224}-2 q^{225}-3 q^{227}-3 q^{228}+3 q^{229}-q^{230}+3 q^{232}-2 q^{233}+q^{234}+2 q^{235}-3 q^{236}+q^{237}+q^{238}-2 q^{239}+3 q^{240}-q^{241}-q^{242}+2 q^{243}
-4 q^{244}-2 q^{245}+2 q^{246}+4 q^{248}+2 q^{249}-3 q^{250}-2 q^{252}-2 q^{253}+3 q^{254}+2 q^{256}+2 q^{257}-3 q^{258}-3 q^{259}-2 q^{260}+q^{261}+4 q^{262}+q^{263}
+q^{264}-q^{265}-3 q^{266}-2 q^{267}+q^{268}+q^{269}+2 q^{270}+q^{271}-q^{272}-q^{273}-q^{274}+q^{275}
\end{math}
}

Theorem \ref{thm.2} implies that:

{\small
\begin{math}
\psi_2(V_{5,7})=V_{0,4}-V_{0,7}+V_{0,10}-V_{0,13}+V_{0,16}-V_{0,19}-V_{1,2}+V_{2,0}+V_{2,6}-V_{2,9}+V_{2,12}-V_{2,15}+V_{2,18}-V_{3,4}+V_{4,2}+V_{4,8}
-V_{4,11}+V_{4,14}-V_{4,17}-V_{5,0}-V_{5,6}+V_{6,4}+V_{6,10}-V_{6,13}+V_{6,16}-V_{7,2}-V_{7,8}+V_{8,0}+V_{8,6}+V_{8,12}-V_{8,15}-V_{9,4}-V_{9,10}+V_{10,2}
+V_{10,8}+V_{10,14}-V_{11,0}-V_{11,6}-V_{11,12}+V_{12,4}+V_{12,10}-V_{13,2}-V_{13,8}+V_{14,0}+V_{14,6}-V_{15,4}+V_{16,2}-V_{17,0}
\end{math}
}
where $V_{n_1,n_2}=V_{n_1 \om_1 + n_2 \om_2}$.

For future checks with other formulas, Theorem \ref{thm.1} implies that
$J_{2,3,70,70}(1/q)$ is a polynomial of $q$ with exponents with respect to $q$
in the interval $[280,30100]$ (where the end points are attained), leading
and trailing coefficients $1$ and
coefficients in the interval $[-55196, 65594]$, where the coefficient
$-55196$ is attained at precisely at $q^{18854}$ and $q^{18925}$ and the 
coefficient $65594$ is attained precisely at $q^{18165}$. In other words,
we have
$$
J_{2,3,70,70}(1/q)=q^{280} + \dots + 65594 q^{18165} + \dots
- 55196 q^{18854} + \dots -55196 q^{18925} + \dots + q^{30100}
$$

Using Theorem \ref{thm.1} it is possible to compute the colored
Jones polynomials $J_{T(2,3),n_1,n_2}(q)$ for $n_1,n_2=0,\dots,100$. 

\subsection{Acknowledgment}
The paper came into maturity following requests for explicit formulas
for the colored Jones polynomial of a knot, during visits of the first author
in the Max-Planck-Institut f\"ur Mathematik in 2009-2010, and during an
Oberwolfach workshop in August 2010. The first author  
wishes to thank R. Lawrence and D. Zagier for their interest,
J. Stembridge for enlightening conversations and the
organizers of the Oberwolfach workshop 1033/2010,
P. Gunnells, W. Neumann, A.S. Sikora and D. Zagier 
for their superb hospitality.

\bibliographystyle{hamsalpha}\bibliography{biblio}
\end{document}